\documentclass[11pt]{amsart}

\usepackage{amsthm} 
\usepackage{amsrefs}
\usepackage{amssymb}
\usepackage{amsmath}
\usepackage{graphicx}									 
\usepackage{stmaryrd}                  
\usepackage{footmisc}                  
\usepackage{paralist}
\usepackage{wrapfig}
\usepackage[draft]{pdfcomment}
\usepackage{bbm}
\usepackage[T2A,T1]{fontenc}
\usepackage[utf8]{inputenc}
\usepackage{lmodern}
\usepackage[ngerman,russian,british]{babel}
\usepackage[arrow, matrix, curve]{xy}
\usepackage{color}
\usepackage{capt-of}
\usepackage{enumitem}

\usepackage{todonotes}
\usepackage{hyperref}
\usepackage{aliascnt}

\newtheorem{thm}{Theorem}[section]

\newaliascnt{lem}{thm}
\newtheorem{lem}[lem]{Lemma}
\aliascntresetthe{lem}

\newaliascnt{cor}{thm}
\newtheorem{cor}[cor]{Corollary}
\aliascntresetthe{cor}

\newaliascnt{prp}{thm}
\newtheorem{prp}[prp]{Proposition}
\aliascntresetthe{prp}

\newaliascnt{thmA}{thm}
\newtheorem*{thmA}{Theorem A}
\aliascntresetthe{thmA}

\newaliascnt{thmB}{thm}
\newtheorem*{thmB}{Theorem B}
\aliascntresetthe{thmB}

\theoremstyle{definition}
\newaliascnt{dfn}{thm}
\newtheorem{dfn}[dfn]{Definition}
\aliascntresetthe{dfn}

\newtheorem*{qst}{Question}

\theoremstyle{remark}

\DeclareMathOperator{\dist}{dist}
\DeclareMathOperator{\A}{Alex}

\title[Toward canonical convex functions in Alexandrov spaces]{Toward canonical convex functions in Alexandrov spaces}
\author{Artem Nepechiy}
\address{Artem Nepechiy, Mathematisches Institut der Universit\"at zu K\"oln, Weyertal 86-90, 50931 K\"oln, Germany}
\email{artem.nepechiy@uni-koeln.de}

\begin{document}

\begin{abstract} 
	We construct for every finite-dimensional Alexandrov space $A$ and every point $p \in A$ a $2$-convex function $f_p$ in a small neighborhood around $p$, which approximates $\dist_p^2$ up to second order. Moreover, the function $f_p$ can be lifted to Gromov-Hausdorff close Alexandrov spaces of the same dimension.
\end{abstract}
\maketitle

\section{Introduction}
An Alexandrov space $A$ is a complete, geodesic metric space space satisfying a synthetic lower sectional curvature bound and having finite Hausdorff dimension. One can introduce
 $\lambda$-convex ($\lambda$-concave) functions, i.e. continuous maps $f:A \rightarrow \mathbb{R}$, such that
	$$ f \circ \gamma(t) - \frac{\lambda}{2}t^2 $$
is convex (concave) for every unit-speed shortest path $\gamma$. Although Alexandrov spaces admit a wide variety of concave functions, i.e  for every point $p \in A$ in a 
Alexandrov space the function
 $\dist_p^2$ ist $(2+O(r^2))$-concave on $B_r(p)$, convex functions are difficult to obtain.\par
The main result of this paper is to construct a map, which approximates $\dist_p^2$ up to second order and has convexity properties as in the Euclidean situation. One of the main results is:

\begin{thmA}\label{Theorem A}
	Let $A$ be a finite-dimensional Alexandrov space and $p \in A$ a point. Then there exist $r>0$ and a locally Lipschitz $2$-convex function $f:B_r(p) \rightarrow \mathbb{R}$ satisfying
		$$ \lim_{x \rightarrow p} \frac{f(x) - \dist^2_p(x)}{\dist^2_p(x)} = 0.  $$

	Moreover, the map $f$ is liftable to $GH$-nearby Alexandrov spaces of the same dimension in the sense of \autoref{Definition: Constructible}.
\end{thmA}
Theorem A provides an affirmative answer to a question asked in \cite{MR2408266}[Question 7.3.6].
	\begin{qst}[\cite{MR2408266}]
		Is it true that for any $p \in A$ and any $\varepsilon>0$, there is a $(-2+\varepsilon)$-concave function $f_p$ defined in a neighborhood of $p$, such that $f_p(p)=0$ and $f_p \geq - \dist_p^2$?
	\end{qst}
In order to prove theorem A, we are going to show the following result first:
\begin{thmB}\label{Theorem B}
	Let $A$ be a finite-dimensional Alexandrov space and $p \in A$ a point. Then for any $\varepsilon > 0$ there exist an $r>0$ and a map $f_{\varepsilon}:B_r(p) \rightarrow \mathbb{R}$ satisfying the following conditions:
		\begin{enumerate}
			\item \label{Main Thm:Property1} The function $f_{\varepsilon}$ is $(-2 + \varepsilon)$-concave and Lipschitz continuous on $B_r(p)$.
			\item \label{Main Thm:Property2} The function $f_{\varepsilon}$ has an isolated maximum at $p$ and satisfies $f_{\varepsilon}(p)=0$.
			\item \label{Main Thm:Property3} For all $x \in B_r(p)$ one has $f_{\varepsilon}(x) \geq -\operatorname{dist}_p^2(x).$
	\end{enumerate}
\end{thmB}
	Although theorem B looks like a corollary of theorem A, it is the other way around. By refining some arguments in the proof of theorem B, one obtains the stronger theorem A.
	Strictly concave functions, meaning $\lambda$-concave functions with $\lambda<0$, have been constructed before in \cite{MR1220498}. Perelman introduced a construction, which could produce liftable functions satisfying (\ref{Main Thm:Property1}) in theorem B. More precisely, let $(A_i^n,p_i)$ be a pointed sequence of $n$-dimensional Alexandrov spaces of curvature $\geq \kappa$ converging in the pointed Gromov-Hausdorff sense to an Alexandrov space $(A^n,p)$ of the same dimension. Then for any $\lambda>0$ there exits $r>0$ such that one has a $(-\lambda)$-concave, $1$-Lipschitz map $f:B_r(p) \rightarrow \mathbb{R}$. Moreover, there exists $N \in \mathbb{N}$, such that for all $i \geq N$ one has $(-\lambda)$-concave, $1$-Lipschitz maps $f_i:B_r(p_i) \rightarrow \mathbb{R}$, which are uniformly close to $f$. By taking the minimum over such functions Kapovitch improved this construction in \cite{MR1904560} to additionally satisfy (\ref{Main Thm:Property2}) in theorem B. However, maps obtained that way might fail to satisfy condition (\ref{Main Thm:Property3}). \par
	Now turn to the strategy for proving theorem B. For every point $p$ in a finite-dimensional Alexandrov space $A$, denote by $o_p$ the apex of the tangent cone $T_pA$, then the function $-\dist_{o_p}^2$ is $(-2)$-concave, since $T_pA$ is an Euclidean cone. Thus it satisfies the conclusions of theorem B. One would like to lift the function from the tangent space to a small neighborhood around $p$, such that all properties, which are promised by theorem $B$, are preserved.\par
%
	In order to define the function $-\dist^2_{o_p}$ on $T_pA$ instead of looking at distances from $o_p$ one can equivalently look at distances to the unit sphere $S_1(o_p)$ around $o_p$. In other words $-\dist^2_{o_p}$ can be defined by knowing the distance to the boundary of the convex set $B_1(o_p)$, the ball of radius $1$ around $o_p$. \par
	Therefore instead of lifting the point $o_p$ one needs to lift the convex set $B_1(o_p) \subset T_pA$ to a convex set close to $B_r(p) \subset A$ with regard to the Hausdorff-distance for very small $r>0$. Notice that in general it is an open question how to lift convex sets to Gromov-Hausdorff close Alexandrov spaces.
	\begin{qst}[\cite{MR2408266}(Question 9.1.3')]
		Assume $A_i \xrightarrow{d_{GH}} A, A_i \in  \A^n(\kappa)$, $\dim(A)=n$ (i.e there is no collapse) and $\partial A= \emptyset$. Let $S \subset A$ be a convex hypersurface. Is it always possible to find a sequence of convex hypersurfaces $S_i \subset A_i$ which converges to $S$?
	\end{qst}
	There is hope that the techniques of our result can be used to answer this question. The special case $(\frac{1}{r}A,p) \rightarrow (T_pA,o_p)$ and $S= S_1(o_p)$ will be solved later in this paper. Suppose one has convex sets $C_r \subset A$ satisfying $d_H(B_r(p),C_r)/r \rightarrow 0$ for $r \rightarrow 0$, where $d_H$ denotes the Hausdorff distance. Assume additionally that the boundary $\partial C_r$ comes from a $(-1 +\varepsilon(r))$-concave function. Then one can define a map by
		$$ f_r: B_r(p) \setminus \overline{B}_{\frac{r}{4}}(p) \rightarrow \mathbb{R}; x \mapsto \varphi \circ \dist_{\partial C_r}(x); \quad \varphi: \mathbb{R} \rightarrow \mathbb{R}; t \mapsto -(t - C)^2, $$
	where $C$ denotes an appropriately chosen constant. The fact, that $\partial C_r$ comes from a $(-1 + \varepsilon(r))$-concave function will imply $(-2 +\varepsilon(r))$-concavity of $f_r$. Indeed it is well known, that for an Alexandrov space $A$ with curvature $\geq 0$ and non-empty Alexandrov boundary $\partial A$ the function $\dist_{\partial A}$ is concave \cite{MR2408266}[Thm. 3.3.1]. This statement and in particular its proof can be generalized to obtain a sharper bound on the concavity of $\dist_{\partial A}$, given that the boundary $\partial A$ comes from a strictly concave function (compare \autoref{Prop: concavity estimates}).\par
	Similar statements have been proven in \cite{MR2595678} for spaces having upper curvature bounds as well as spaces having bounded curvature from below. We provide an alternative shorter proof in the case of a lower curvature bound.
	\par
	Since $C_r$ is Hausdorff close to $B_r(p)$, the map $f_r$ will satisfy the following inequalities
		\begin{equation}\label{Eq: Schranken} - \dist_p^2(x) \leq f_r(x) \leq -(1-\varepsilon(r)) \dist_p^2(x) \text{ for } x \in B_r(p) \setminus B_{\frac{r}{4}}(p).\end{equation}
	Assume one has two $(-2+\varepsilon(r))$-concave maps
		$$f_1 : B_r(p)  \rightarrow \mathbb{R}, \quad  f_2 : B_{\frac{r}{2}}(p)  \rightarrow \mathbb{R},$$
	coming from the above construction and satisfying inequalities from (\ref{Eq: Schranken}). Consider $F_1:=\min \lbrace \varphi \circ f_1,f_2 \rbrace$ whenever the minimum is defined, if $\varphi: \mathbb{R}\rightarrow \mathbb{R}$ is chosen properly, one can achieve $F_1= f_1$ on $B_r(p) \setminus \overline{B}_{r (1-\varepsilon)}(p)$ and $F_1 = f_2$ on $B_{\frac{r}{4}}(p)$. An inductive argument gives the function promised by theorem B. This is, what is called \emph{self-improvement} of the function.\par
	The bread and butter of the construction in the proof of theorem B is, that it can be lifted to nearby Alexandrov spaces. This means, if $(A_i,p) \rightarrow (A,p)$ in $\operatorname{Alex}(\kappa)$ without collapse and $f:B_r(p)  \subset A \rightarrow \mathbb{R}$ is a function then there exist $f_i:B_r(p_i) \subset A_i \rightarrow \mathbb{R}$ with similar properties. In the upcoming definition it will be made precise, which properties the lifts should preserve.
	\begin{dfn}[$\varepsilon$-Constructible]\label{Definition: Constructible}
		Fix $\varepsilon,R>0$ and let $p \in A^n \in \operatorname{Alex}^n(\kappa)$ be a point in an $n$-dimensional Alexandrov space. Assume $f_{\varepsilon}:B_R(p)\rightarrow \mathbb{R}$ is a function satisfying the conclusions of theorem B for $\varepsilon>0$.\par
		The map $f_{\varepsilon}$ is called $\varepsilon$-constructible if for any sequence $A^n_i \in \operatorname{Alex}^n(\kappa)$ with $(A_i^n,p_i)\rightarrow (A^n,p)$, there exists $N \in \mathbb{N}$ such that for all $i \geq N$ there is a function $f_i:B_R(p_i) \rightarrow \mathbb{R}$, which satisfies the following conditions:

		\begin{enumerate}
			\item The function $f_i$ is $2R$-Lipschitz and $(-2+\varepsilon)$-concave on $B_R(p_i)$.
			\item The function $f_i$ satisfies $f_i(p_i)=0$ and
				$$ -\dist_{p_i}^2(x) \leq f_i(x) \leq -(1-2\varepsilon) \dist^2_{p_i} (x)  \text{ for all } x \in B_R(p_i) \setminus B_{\aleph(i) \cdot R}(p_i), $$
			where $\aleph(i)$ denotes a sequence satisfying $\aleph(i) \rightarrow 0$ for $i \rightarrow \infty$.
		\end{enumerate}
	\end{dfn}
	It remains to explain, how the convex sets $C_r$, mentioned above, are obtained. With the above definition for given $n \in \mathbb{N}$ one can prove a series of propositions $P(n,k)$ for $1 \leq k \leq n$, which are crucial for the construction of $C_r$.
	\begin{prp}[$P(n,k)$]\label{Main Proposition}\label{Proposition}
		Let $A$ be an Alexandrov space of dimension $n$ without boundary and $p \in A$. Denote by $T_pA$ the tangent space at $p$. Fix arbitrary $\varepsilon>0$ and $0 \leq k \leq n$. \par
		If $T_pA=\mathbb{R}^k \times T_bB$ for some point $b$ in an Alexandrov space $B$ of dimension $n-k$ without boundary, then there exists $R>0$ and  a function $f:B_R(p) \subset A \rightarrow \mathbb{R}$ satisfying:
			\begin{enumerate}
				\item  The function $f$ is $2R$-Lipschitz and $(-2 + \varepsilon)$-concave on $B_R(p)$.
				\item The function $f$ has an isolated maximum at $p$ and satisfies $f(p)=0$.
				\item For all $x \in B_R(p)$ one has $f(x) \geq -\operatorname{dist}_{p}^2(x).$
				\item The function $f$ is $\varepsilon$-constructible.
			\end{enumerate}
		Denote the statement of the Proposition for fixed $n$ and $k$ by $P(n,k)$.
	\end{prp}
		Denote the statement of theorem B for dimension $n$ by $T(n)$. The assertion of $T(n)$ will be proven by backward induction using $P(n,k)$, where the induction scheme is given by
			$$ P(n,n) \Rightarrow P(n,n-1) \Rightarrow \ldots \Rightarrow P(n,1) \Rightarrow P(n,0) \Rightarrow T(n).$$
		Let us illustrate the idea in dimension $2$. The Statement $P(2,2)$ means that we start with a regular point. Consider an $(n,0)$-strainer $\lbrace(p_j,q_j)\rbrace$ in $T_pA= \mathbb{R}^n$ and define the map $ f:= \sum \varphi  \circ \dist_{p_j}$ for a carefully chosen function $\varphi: \mathbb{R} \rightarrow \mathbb{R}$. It will be shown, that this map is $\varepsilon$-constructible and the lifts are obtained in the obvious way, i.e. $f_i=  \sum \varphi  \circ \dist_{p_{ij}}$, where $p_{ij}$ denotes the canonical lift of $p_j$ using Hausdorff approximations. Applying the self-improvement argument yields $P(n,n)$.\par
		If $A$ is an two-dimensional Alexandrov space without boundary the statements $P(2,2)$ and $P(2,1)$ coincide. It remains to prove the implication $P(2,1) \Rightarrow P(2,0)$. For a point $q \in T_pA$ lying on the unit sphere one locally has
			\begin{equation}
				- \dist^2_{o_p}= 1 +2 B_{\gamma}(x) - \dist_q^2,
			\end{equation}
		where $\gamma$ is the ray starting at the origin and going through $q$ and $B_{\gamma}$ denotes the Busemann function associated to the ray $\gamma$. Since a shortest path, namely the ray $\gamma$ goes through $q$, by the splitting theorem \cite{MR0256327} the tangent space $T_qT_pA$ splits of an additional $\mathbb{R}$-factor, thus $P(2,1)$ is applicable. Now the right hand side of this equation can be approximated by a distance function from a point lying on $\gamma$ and being sufficiently far away from the origin and the function coming from $P(2,1)$. \par
		Using this argument one can produce for every $q \in B_1(o_p) \setminus B_{1/2}(o_p)$ a  $(-2+\varepsilon)$-concave function $F_i:B_{r_q}(q) \rightarrow \mathbb{R}$ such that
			\begin{equation}\label{Equation: Well defined Condition}
				F_i(x) < \dist^2_{o_p}(x) \text{ for all } x \in B_{r_q/10}(q) \text{ and } 
				F_i(x) > \dist^2_{o_p}(x) \text{ for all } x \in B_{r_q/2}(q).
			\end{equation}
		Using compactness of $ B_1(o_p) \setminus B_{1/2}(o_p)$ one obtains a finite covering $\lbrace B_{r_i}(q_i) \rbrace_{i=1}^N$, such that $\lbrace B_{r_i/20}(q_i) \rbrace_{i=1}^N$ is still a covering together with functions $F_i:B_{r_i}(q_i) \rightarrow \mathbb{R}$ as above. Now the conditions in (\ref{Equation: Well defined Condition}) imply that
			$$ F:= \min_{1 \leq i \leq N } F_{i}(x) \cdot \mathbbm{1}_i(x),$$
		where $\mathbbm{1}_i$ denotes the characteristic function for $B_{r_i}(q_i)$, is a well-defined, $(-2+\varepsilon)$-concave, Lipschitz function on all of $ B_1(o_p) \setminus B_{1/2}(o_p)$. Finally the level set of $\lbrace F=1 \rbrace$ bounds a convex set in $T_pA$.\par
		The arguments are designed in a way that they can be repeated verbatim for all sufficiently close lifts of the $F_i$. In particular there is a way to lift the function $F$ to a small annulus around $p$ in $A$. The level sets of the lifts then bound a convex region $C_r$ in $A$. In view of the remarks made earlier, this finishes the proof of $P(2,0)$.\par
		In dimensions higher than $2$, one needs a more sophisticated argument. Let us illustrate this in the situation $n=3$. $P(3,3)$ and $P(3,2)$ can be proved the same way as before. The statement $P(3,1)$ includes tangent spaces like $\mathbb{R} \times C^2_{\alpha}$, where $C^2_{\alpha}$ denotes a two-dimensional cone with opening angle $\alpha$. The issue is that one cannot use $P(3,2)$ at points $\mathbb{R} \times \lbrace o \rbrace$, where $o$ denotes the apex of $C^2_{\alpha}$, since at such points the tangent space $T_q T_pA$ does not split an additional $\mathbb{R}$-factor as before.\par
		The solution is to no longer approximate $-\dist_{o_p}^2: \mathbb{R}^k \times T_bB \rightarrow \mathbb{R}$ but to introduce a new function
			$$f_c: \mathbb{R}^k \times T_bB \rightarrow \mathbb{R}; (x,y) \mapsto  - \frac{1}{2} \cdot \left(  \Vert x  \Vert^2 + \left( \vert o_{b} y \vert_{B} +c \right)^2  \right) $$
		for sufficiently small $c>0$. The advantage of that function is, that it can be approximated at points $\mathbb{R}^k \times \lbrace o \rbrace$ using weaker functions than the ones coming from $P(n,k+1)$. This weak map is the sum of $H_1 : \mathbb{R}^k \rightarrow \mathbb{R}, H_2: T_bB \rightarrow \mathbb{R}$, where $H_1$ is constructed similar to the step $P(k,k)$ and $H_2$ is the map used by Kapovitch in \cite{MR1904560}. Since both maps are defined in terms of distance functions they can naturally be extended to $\mathbb{R}^k \times T_bB$. It remains to justify that the sum is $\varepsilon$-constructible, which will be carried out in \autoref{subseq: second order with factor}. This finishes the proof of theorem B and thus of theorem A, when the Alexandrov space has no boundary. The case with boundary can be immediately deduced from this (compare \autoref{Cor: Boundary case}).\par
		This paper is a condensed version of the authors thesis \cite{Nepechiy2018}, where the arguments can be found in full detail.\par
		\textbf{Acknowledgements.} I  would like to thank my advisor Burkhard Wilking for this guidance and support during the PhD thesis and Alexander Lytchak for useful comments regarding this paper.\par
		The author was partially supported by the DFG grant SPP 2026.
\section{Preliminaries}
	We assume familiarity with Alexandrov spaces, in particular with \cite{MR1835418},\cite{MR1185284},\cite{MR1886682},\cite{MR1320267}. In this section we will fix notation and collect some less known facts and definitions, which will be needed later on. \par
	By $A^n$ we usually denote an Alexandrov space of dimension $n$, that is a metric space satisfying the Toponogov triangle comparison. For $x,y \in A$ the distance between $x$ and $y$ is denoted by $\vert xy \vert$. If $C,D \subset A$ then $\vert C D \vert$ is the infimum over $\vert c d \vert$, where $c \in C$ and $d \in D$. For $\lambda>0$ by $\lambda A$ we denote the $\lambda$-rescaled space, meaning that the metric is given by $\vert xy \vert _ {\lambda \cdot A} = \lambda \cdot \vert xy \vert_A$. An open ball of radius $r$ is denoted by $B_r(p)$, the closed ball is denoted by $\overline{B}_r(p)$. The tangent space at $p \in A$ is denoted by $T_pA$, for the space of directions we write $\Sigma_p$. Elements in $\Sigma_p$ are denoted by arrows. So an element in $\Sigma_p$ representing a shortest path from $p$ to $q$ is denoted by $\uparrow_p^q$, the set of all directions from $p$ to $q$ is denoted by $\Uparrow_p^q$.\par
	\begin{dfn}[$\lambda$-concavity]\label{Def: lambda concavity}
	Denote by $A$ an $n$-dimensional Alexandrov space without boundary and by $\Omega \subset A$ an open set. A locally Lipschitz function
	\[f: \Omega \rightarrow \mathbb{R}\]
	is called $\lambda$-concave on $\Omega$, if for all $x,y \in \Omega$ and each unit-speed shortest path $\gamma$ lying in $\Omega$ and connecting $x$ and $y$ the function
	\[ f \circ \gamma (t) - \frac{\lambda}{2} t^2 \]
	is concave on its domain of definition.
\end{dfn}
\begin{dfn}[$\lambda$-concavity for spaces with boundary]\label{Def: Concavity for spaces with boundary}
	Denote by $A$ an $n$-dimensional Alexandrov space with boundary and by $\Omega \subset A$ an open set. A locally Lipschitz function
	\[f: \Omega \rightarrow \mathbb{R}\]
	is called $\lambda$-concave on $\Omega$, if $f \circ p$ is $\lambda$-concave on $p^{-1}(\Omega) \subset \operatorname{Doub}(A)$, where $\operatorname{Doub}(A)$ denotes the doubling of $A$ and $p: \operatorname{Doub}(A) \rightarrow A$ is the canonical projection.
\end{dfn}
	From the definition of $\lambda$-concavity it is clear that it is sufficient to prove theorem A and theorem B for spaces without Alexandrov-boundary.
		\begin{cor}[Boundary case]\label{Cor: Boundary case}
			If theorem A, theorem B are true for all Alexandrov spaces with empty boundary, then they also hold for arbitrary Alexandrov spaces.
		\end{cor}
			\begin{proof}
				Consider the doubling $D(A)$ of $A$, this space is an Alexandrov space of dimension $n$ without boundary. It comes with a canonical isometric involution $I: D(A) \rightarrow D(A)$, which interchanges the points of the first copy of $A$ with points in the second copy of $A$. \par Theorem A provides a map $f:B_r(p) \subset D(A) \rightarrow \mathbb{R}$. Define $g$ by $f \circ I$. \par
				Obviously $g$ has the same properties as $f$ in regard to theorem A. Thus the minimum $\min(f,g)$ is a map satisfying all conclusions of the above theorems but it is invariant under the canonical projection from \autoref{Def: Concavity for spaces with boundary}. This finishes the proof.
			\end{proof}

		\begin{dfn}[Quasigeodesics]
	A curve $\gamma$ in an Alexandrov space $A$ is called quasigeodesic  if for any $\lambda \in \mathbb{R}$, given a $\lambda$-concave function $f$ the map $f \circ \gamma$ is $\lambda$-concave.
\end{dfn}

\begin{thm}[Existence of quasigeodesics,\cite{MR2408266}]\label{Thm: Existence of Quasigeodesics}
	Let $A$ be an Alexandrov space of finite dimension, then for any point $x \in A$ and any direction $\xi \in \Sigma_x$ there exists a quasigeodesic $\gamma: \mathbb{R}\rightarrow A$ such that $\gamma(0)=x$ and $\gamma^+(0) = \xi$.\par
	Moreover for $x \in \partial A$ and $ \xi \in \Sigma_x \partial A$ the quasigeodesic $\gamma$ can be chosen to lie completely in $\partial A$, where $\Sigma_x \partial A$ is defined by
	$$ \Sigma_x \partial A := \left \lbrace \uparrow \in \Sigma_x A \, \vert \, \text{ There ex. a sequence } y_i \in \partial A \text{ such that } \uparrow_x^{y_i} \rightarrow \uparrow  \right \rbrace. $$
\end{thm}

\section{Regular case}\label{sec: reg case}
	The goal of this section is to establish the induction base for the proof of theorem B and theorem A, that is to prove $P(n,n)$ of \autoref{Main Proposition} and to introduce the self-improvement procedure described in \autoref{subsec:Improvement}.\par
	Assume $T_pA = \mathbb{R}^n$. The idea is to write down a model function $f:=\sum_{j=1}^{N} \varphi_\varepsilon \circ \dist_{q_j}$, for carefully chosen $\varphi_{\varepsilon}: \mathbb{R}\rightarrow \mathbb{R}$, satisfying  the conclusion of theorem B. For such maps there is an obvious way to construct lifts. It remains to proof that the lifts also satisfy the conclusions of theorem B, this is carried out in \autoref{Lem: Concavity of lifts}.
\subsection{Model-construction and lifts}\label{subseq Model constrution}
	\begin{lem}[Model function]
		Denote by $\dist_{p_i}$ the distance function from $p_i$, where $p_i \in \mathbb{R}^n$ are given by $p_i:=e_i,p_{-i}:=-p_i$ for $1 \leq i \leq n$. Set
			$$ \mu_R: B_R(0) \rightarrow \mathbb{R}; x \mapsto \sum_{ \genfrac{}{}{0pt}{}{i=-n}{i \neq 0} }^n - \frac{1}{2}\left[ (1+R)  - \dist_{p_i}(x)\right]^2 +  \frac{ R^2}{2}$$
		Then for every $\varepsilon>0$ there exits $\overline{R}>0$ such that the function $\mu_R$ satisfies the conclusion of theorem B for all $0<R<\overline{R}$.
	\end{lem}
	\begin{proof}
			Given a point $p \in \mathbb{R}^n$ and the distance function $\dist_p$. Using the abbreviation $p_e = p/ \Vert p \Vert$, its Taylor series up to order three is given by
				$$ \dist_p(x)  = \Vert p \Vert -\langle x, p_e \rangle +\frac{1}{2} \left( \frac{\Vert x \Vert^2}{\Vert p \Vert} - \frac{\langle x,p_e \rangle^2}{\Vert p \Vert} \right)+ \frac{1}{2} \left(   \frac{\Vert x \Vert ^2 \langle x, p_e \rangle}{\Vert p \Vert^2} - \frac{\langle x, p_e \rangle^3}{\Vert p \Vert^2} \right).$$
			Using direct calculations one obtains
				$$\mu_R(x)=-\Vert x \Vert^2 \left(1 - R(n-1) \right) + o(\Vert x \Vert^3),$$
			which implies the result.
	\end{proof}
		Denote by $\theta_r: B_1(o_p) \subset T_pA \rightarrow B_1(p) \subset  \frac{1}{r}A$ the Hausdorff-approximations coming from the convergence $(\frac{1}{r}A,p) \rightarrow (T_pA,o_p)$.
	\begin{dfn}[Lift of the model function]\label{Def: r-Lift}
		Assume $T_pA=\mathbb{R}^n$. Then once can define the $r$-lift $f_r:B_R(p) \subset \frac{1}{r} A \rightarrow \mathbb{R}$ of $\mu_R: B_R(0) \subset \mathbb{R}^n \rightarrow \mathbb{R}$ by the formula
			$$f_r(x)= \sum_{ \genfrac{}{}{0pt}{}{i=-n}{i \neq 0} }^n - \frac{1}{2}\left[ (1+R)  - \dist_{\theta_r(p_i)}(x)\right]^2 +  \frac{ R^2}{2}.  $$
	\end{dfn}
	The proof that the model function is $(-2+\varepsilon)$-concave highly relies on the structure of $\mathbb{R}^n$ and on the differentiability of distance functions. These methods clearly do not carry over to the $r$-lifts. Hence one needs to give a new proof for the $(-2+\varepsilon)$-concavity of the model function, which uses only tools of Alexandrov geometry. This is carried out in the next lemma.
	\begin{lem}[Concavity of lifts]\label{Lem: Concavity of lifts}
		For every $\varepsilon>0$ there exists $R, \overline{r} >0$ such that for all $r \leq \overline{r}$ the $r$-lift $f_{r}$ defined in \autoref{Def: r-Lift} is $(-2+\varepsilon)$-concave on $B_R(p) \subset \frac{1}{r}A$.
	\end{lem}
		\begin{proof}
			The functions $f_r$ are clearly continuous, thus in order to prove $(-2+\varepsilon)$-concavity it is sufficient to check
				\begin{equation}\label{Eq: Concavity inequality}
				2 f_r(m) - f_r(x) - f_r(y) \geq   - \frac{-2 + \varepsilon}{4}  \cdot \vert xy \vert_{\frac{1}{r}}^2 =  \left(1-\frac{\varepsilon}{2}\right)  \vert xy \vert_{\frac{1}{r}} ^2 \cdot \frac{1}{2}
				\end{equation}
			for every $x,y \in B_R(p) \subset \frac{1}{r}A$ and every midpoint $m$ between $x$ and $y$.\par
			Set
				$$ f_{r,i}(x):= - \frac{1}{2 } \left( (1+R)-\vert x \theta_r(p_i) \vert_{\frac{1}{r}} \right)^2+ \frac{R^2}{2}. $$
			One can bound $	2f_{r,i}(m) - f_{r,i}(x) - f_{r,i}(y)$ below, using similar arguments as in the proof of \cite{MR1904560}[Lemma 4.2], by
				$$ - \frac{\varepsilon }{8n}  \cdot \frac{\vert xy \vert^2}{2}  + \left[ \cos^2 \left( \vert \uparrow_m^x \Uparrow_m^{\theta_r(p_i)} \vert \right)  +\cos^2 \left( \left\vert \uparrow_m^y \Uparrow_m^{\theta_r(p_i)} \right\vert \right) \right] \frac{\vert xy \vert^2}{8},  $$
			where $\vert \uparrow_m^x \Uparrow_m^{\theta_r(p_i)} \vert$ denotes the angle between the direction of a shortest path from $m$ to $x$ and all directions of shortest paths from $m$ to $\theta_r(p_i)$. Summation over $i$ gives
				$$ 2f_r(m)-f_r(x) -f_r(y) \geq  - \frac{\varepsilon}{2}\cdot \frac{\vert xy \vert^2}{2} + \sum_{ \genfrac{}{}{0pt}{}{i=-n}{i \neq 0}}^n \left[ \cos^2 \left( \vert \uparrow_m^x \Uparrow_m^{\theta_r(p_i)} \vert \right)  +\cos^2 \left( \left\vert \uparrow_m^y \Uparrow_m^{\theta_r(p_i)} \right\vert \right) \right]\cdot  \frac{\vert xy \vert^2}{2}.   $$
			The expression in the brackets can be estimated using \autoref{Lemma:BGP}, thus finishing the proof.
		\end{proof}
		In order to state \autoref{Lemma:BGP} properly, one needs to recall the Definition of an $(n,\delta)$-explosion. That is the set of directions in $\Sigma_p$ coming from an $(n,\delta)$ strainer around $p$.
		\begin{dfn}[Explosion]\label{Def: (n,d) explosion}
		Let $\Sigma$ be an Alexandrov space of dimension $(n-1)$ and curvature $\geq 1$. A collection $A_1, \ldots, A_k , B_1 , \ldots , B_k \subset \Sigma$ of compact subsets, $1 \leq k \leq n$ satisfying
			$$ \vert A_i B_i \vert \geq \pi -\delta \text { for all } 1 \leq i \leq k $$
		and
			$$ \vert A_i A_j \vert, \vert B_i B_j \vert , \vert A_i B_j \vert  \geq \frac{\pi}{2} - \delta \text{ for all } 1 \leq i, j \leq k \text { whenever } i \neq j  $$
		is called a $(k,\delta)$-explosion and denoted by $(A_i,B_i)_{i=1}^k$.
	\end{dfn}
	\begin{lem}[\cite{MR1185284}(Lemma 9.3)]\label{Lemma:BGP}
		Let $\Sigma$ be an Alexandrov space of dimension $n-1$ and curvature $\geq 1$. If  $(\Uparrow_m^{s_i},\Uparrow_m^{s_{-i}})_{i=1}^n$ is an $(n,\delta)$-explosion, then for any $\uparrow \in \Sigma$ one has
			$$ \left \vert  \sum_{i=1}^n \cos^2 \left( \vert \uparrow \Uparrow_m^{s_i} \vert  \right) -1 \right \vert \leq \varepsilon(\delta). $$
			Here $\varepsilon$ denotes a real valued function $\varepsilon:\mathbb{R}^+ \rightarrow \mathbb{R}^+$ satisfying $\varepsilon(\delta) \rightarrow 0$ for $\delta \rightarrow 0$.
	\end{lem}
	This gives us the desired building blocks for the upcoming self-improvement procedure. Indeed since $f_r \rightarrow \mu_R$ for $r \rightarrow 0$ for sufficiently small $r$ all $r$-lifts $f_r$ satisfy (\ref{Eq: Schranken}) on $B_R(p) \setminus B_\frac{R}{4}(p)$ in $\frac{1}{r}A$. On one hand this property is preserved if one considers the rescaling
		$$g_r: B_{R\cdot r}(p) \subset A \rightarrow 0, x \mapsto r^2 \cdot f_r(x).$$
	On the other hand by (\ref{Eq: Concavity inequality}) $g_r$ is also a $(-2+\varepsilon)$-concave function. This proves:
	\begin{cor}[Self-improvement assumption]\label{Corollary: Self-improvement assumption}
		Let $A$ be an $n$-dimensional Alexandrov space and $p \in A$ a regular point. For every $\varepsilon>0$ there exists $R>0$ and a sequence $(f_i)_{i \in \mathbb{N}}$ of $4nR/2^{i-1}$-Lipschitz, $(-2+\varepsilon)$-concave functions
			$$f_i :B_{\frac{R}{2^{i-1}}}(p) \subset A \rightarrow \mathbb{R}$$
		satisfying
		$$ -\vert  xp \vert^2 \leq f_i(x) \leq -(1-\varepsilon)\vert xp \vert^2  \text{ for all } x \in B_{\frac{R}{2^{i-1}}}(p) \setminus B_{\frac{R}{ 2^{i+1}}}(p) \subset A.$$
	\end{cor}
	\subsection{Self-improvement procedure}\label{subsec:Improvement}~\par
		Starting from the situation of \autoref{Corollary: Self-improvement assumption}, it will be explained how to construct functions promised by theorem B and theorem A. Assume without loss of generality that $R=1$, otherwise do an appropriate rescaling.\par
		The key idea to constructing $F$ as in theorem B, theorem A is taking $f_1,f_2$ as in \autoref{Corollary: Self-improvement assumption} and considering the map $ F_1:=\min\lbrace \varphi \circ f_1 , f_2 \rbrace$ for a real valued function $\varphi: \mathbb{R} \rightarrow \mathbb{R}$. If $\varphi$ is chosen properly, $\varphi \circ f_1$ will still be $(-2+\varepsilon)$-concave and one will have $F_1=\varphi\circ f_1$ on $B_1(p) \setminus B_{1-\varepsilon}(p)$, $F_1= f_2$ on $ B_{\frac{1}{4}}(p)$ and $F_1$ will satisfy
			$$-(1+\varepsilon)\vert  xp \vert^2 \leq F_1(x) \leq -(1-\varepsilon)\vert xp \vert^2  \text{ for all } x \in B_{1}(p) \setminus B_{\frac{1}{8}}(p) \subset  A.$$
		An inductive argument will conclude the proof. The next Lemma will specify the desired properties of the reparametrization function $\varphi$.
	\begin{lem}[Reparametrization function]\label{Le: Reparametrization fucntion}
		Fix $\delta:=\frac{1}{100}$ then for every $\varepsilon>0$ there exist a function $\varphi_{\varepsilon}: [-1,0] \rightarrow \mathbb{R}$ such that
		\begin{enumerate}
			\item \label{Cond: Reparemtrization function is C2} The function  $\varphi_{\varepsilon}$ is two times continuously differentiable.
			\item \label{Cond: Reparemtrization function extends identity} The function $\varphi_{\varepsilon}$ satisfies
				\[
					\varphi_{\varepsilon}(x) =
						\begin{cases}
							(1-\varepsilon) \cdot x, &  \text{ for } -\left( \frac{1}{4} + \delta \right)^2  \leq  x   \leq  0 \\
							\frac{1}{1-\varepsilon} \cdot x, & \text{ for }   - \left( \frac{1}{2} + \delta \right)^2   \leq  x  \leq - \left( \frac{1}{2} - \delta \right)^2\\
							x, &  \text{ for }    x   \leq  -(1-\delta)^2 \\
						\end{cases}.
				\]
			\item \label{Cond: Reparametrization function has universal bound on derivatives} There exists a bound $0<B$, independent of $\varepsilon$, such that for every $x \in[-1,0]$
			\begin{align*}
				1 -  \varepsilon \cdot B &\leq \varphi'_{\varepsilon}(x) \leq 1+  \varepsilon   \cdot B, \\
				- \varepsilon \cdot B & \leq \varphi''_{\varepsilon}(x) \leq \varepsilon \cdot B , \\
				\left( 1 + \varepsilon \cdot B \right)  \cdot x & \leq \varphi_{\varepsilon}(x)  \leq \left( 1 - \varepsilon \cdot B \right) \cdot x .
			\end{align*}
		\end{enumerate}
	\end{lem}
		\begin{proof}
			Extend $\varphi_{\varepsilon}$ by polynomials $p_{\varepsilon,1}$, $p_{\varepsilon,2}$ of degree $5$ to a smooth function.

	\begin{minipage}{\textwidth -15pt}
		\begin{center}
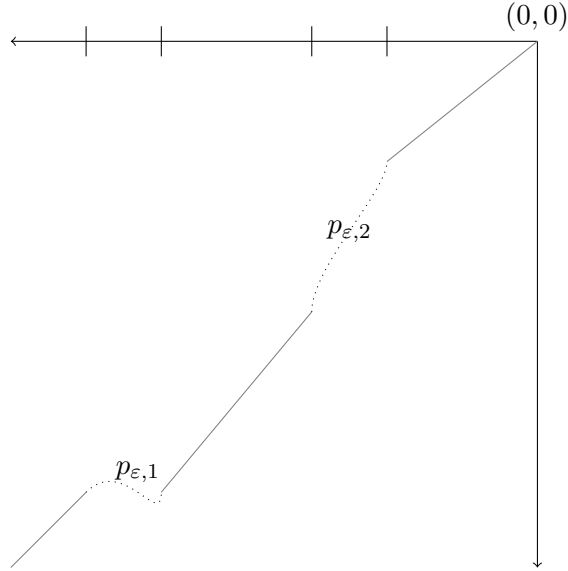

	\begin{tikzpicture}
\draw[->] (0,0) -- (-7,0);
\draw[->] (0,0) -- (0,-7);
\draw (0,0) node[above] {$(0,0)$};
\draw (-2,0.2) -- ++ (0,-0.4) ;
\draw (-3,0.2) -- ++ (0,-0.4) ;
\draw (-5,0.2) -- ++ (0,-0.4) ;
\draw (-6,0.2) -- ++ (0,-0.4) ;
\draw[color=black!50,domain=-2:0] plot (\x,0.8*\x);
\draw[color=black!50,domain=-5:-3] plot (\x,1.2*\x);
\draw[color=black!50,domain=-7:-6] plot (\x,\x);
\draw[dotted] (-6,-6)  .. controls (-5.5,-5.5) and (-5, -6.5) .. node[above]{$p_{\varepsilon,1}$}(-5,-6);
\draw[dotted] (-3,-3.6)  .. controls (-3,-3) and (-2, -2) .. node[align=center]{$p_{\varepsilon,2}$} (-2, -1.6);
\end{tikzpicture}
\captionof{figure}{Extension of  $\varphi_{\varepsilon}$ by polynomials}
		\end{center}
	\end{minipage}

			The conditions (\ref{Cond: Reparemtrization function is C2}) and (\ref{Cond: Reparemtrization function extends identity}) of \autoref{Le: Reparametrization fucntion} are satisfied by construction. Straightforward (although tedious) computation of the polynomials $p_{\varepsilon,1}, p_{\varepsilon,2}$ gives (\ref{Cond: Reparametrization function has universal bound on derivatives}). The full computation can be found in \cite{Nepechiy2018}.
		\end{proof}
	\begin{cor}[Gluing of functions]\label{Cor:Gluing of functions}
		Let $A$ be an $n$-dimensional Alexandrov space and $p\in A$ a point. Fix an arbitrary $\varepsilon>0$ and let $f_1:B_1(p) \rightarrow \mathbb{R}$, $ f_2:B_{\frac{1}{2}}(p) \rightarrow \mathbb{R}$ be $4n$-Lipschitz, $(-2+\varepsilon)$-concave functions satisfying
		\begin{align*}
		-\vert xp \vert^2 \leq f_1(x) \leq  -\left(1-\frac{\varepsilon}{2}\right)\vert xp \vert^2 \quad \text{ for all } x \in B_1(p) \setminus B_{\frac{1}{4}}(p),\\
		-\vert xp \vert^2 \leq f_2(x) \leq  -\left(1-\frac{\varepsilon}{2}\right)\vert xp \vert^2 \quad \text{ for all } x \in B_{\frac{1}{2}}(p) \setminus B_{\frac{1}{8}}(p) .
		\end{align*}
		
		Then there exists a function $F:B_1(p) \rightarrow \mathbb{R}$, which fulfills the following list of conditions:
		\begin{enumerate}
			\item $F$ is locally Lipschitz continuous and $(-2+\varepsilon+\varepsilon \cdot B)$-concave.
			\item The function $F$ satisfies $ F=f_1 \text{ on }  B_1(p) \setminus B_{1-\frac{1}{200}}(p)$.
			\item The function $F$ satisfies $  F= f_2 \text{ on }  B_{\frac{1}{4}+\frac{1}{100}}(p) $.
			\item The function $F$ satisfies
			$$-(1-\varepsilon B ) \vert xp \vert^2 \geq F(x)\geq -(1 +\varepsilon B) \vert x p\vert^2 \text{ for } x \in B_1(p) \setminus B_{\frac{1}{8}}(p),$$
			where in the above $B>0$ is a constant not depending on $\varepsilon$ or the functions $f_1,f_2$.
		\end{enumerate}
	\end{cor}
		\begin{proof}
			Set
				$$F=\begin{cases}
						\varphi_{\varepsilon'} \circ f_1, & x \in B_1(p) \setminus B_{\frac{1}{2}}(p) \\
						\min \lbrace \varphi_{\varepsilon'} \circ f_1 , f_2 \rbrace, & x \in B_{\frac{1}{2}}(p) \setminus B_{\frac{1}{4}}(p) \\
						f_2, &  x \in B_{\frac{1}{4}}(p)
			\end{cases},$$
		for appropriate $\varepsilon'>0$ and $\varphi_{\varepsilon'}$ as in \autoref{Le: Reparametrization fucntion}. By construction of $\varphi_{\varepsilon'}$, namely (\ref{Cond: Reparemtrization function extends identity}) in \autoref{Le: Reparametrization fucntion}, $F$ satisfies conditions (ii)-(iv) in \autoref{Cor:Gluing of functions}. Observe that $F$ coincides with $\varphi_{\varepsilon'} \circ f_1$ in a neighborhood of the boundary of $B_{\frac{1}{2}}(p)$. Thus for any point there exists a small neighborhood such that $F$ is equal to $\varphi_{\varepsilon'} \circ f_1, \min \lbrace \varphi_{\varepsilon'} \circ f_1 , f_2 \rbrace, f_1$ or $f_2$. Clearly each of these functions is locally Lipschitz. Since the minimum of $\lambda$-concave maps is $\lambda$-concave, in order to check $(-2+\varepsilon+\varepsilon \cdot B)$-concavity of $F$, it is sufficient to show concavity for $\varphi_{\varepsilon'} \circ f_1$.\par
		Fix a unit-speed shortest path $\gamma:[a,b] \rightarrow A$. Then $(-2+\varepsilon)$-concavity of $f_1$ is equivalent to the existence of a local smooth support function (compare \cite{article}[Lemma 2.3]), that is for every $x \in [a,b]$ there exists a neighborhood $U_x$ and a twice differentiable function $g:U_x \rightarrow \mathbb{R}$ such that
			$$ f_1(x) = g(x), \quad  f_1(z) \leq g(z), \quad g''(z) \leq -2+\varepsilon \qquad \forall z \in U_x. $$
		Clearly $\varphi_{\varepsilon'} \circ g$ is a local support function for $\varphi_{\varepsilon'} \circ f_1$ satisfying
			\begin{align*}
			(\varphi_{\varepsilon} \circ g(z))''&  =\varphi''_{\varepsilon} \circ g(z) \cdot (g'(z))^2 + \varphi'_{\varepsilon} \circ g(z) \cdot g''(z)\\
			& \leq  \varepsilon {B} \cdot   (g'(z))^2 + (1- \varepsilon \cdot  {B}) \cdot (-2+\varepsilon),
			\end{align*}
			where $B$ is the bound appearing in \autoref{Le: Reparametrization fucntion} and $g'(z)$ is bounded by the Lipschitz constant of $f_1$. This proves the claim.
		\end{proof}
	Now we are ready to give the proof of $P(n,n)$ of \autoref{Proposition}.
	\begin{proof}[Proof of $P(n,n)$]
		Assume we have the situation of \autoref{Corollary: Self-improvement assumption} with $R=1$. Set $F_0:=f_1$ and construct $F_1:= F(1,2):B_1(p) \rightarrow \mathbb{R}$, where $F(1,2)$ is the map coming from \autoref{Cor:Gluing of functions} applied to the maps $f_1,f_2$. Now one can define a map $F_2:B_1(p) \rightarrow \mathbb{R}$ by setting $F_2\vert_{B_1(p) \setminus B_{\frac{1}{2}}(p)}=F_1$ and $F_2 \vert_{B_{\frac{1}{2}}(p)} = F(2,3)$, where $F(2,3)$ is the map coming from \autoref{Cor:Gluing of functions} applied to the maps $f_2,f_3$ after rescaling (notice that here we used that the Lipschitz constant of the $f_i$ scales with the domain of definition).\par
		Inductively one can construct a map $F_{\infty}:B_1(p) \rightarrow \mathbb{R}$ satisfying the conclusions of theorem B. Notice that this map is also $\varepsilon$ constructible, since finitely many gluing steps can be carried out on sufficiently close Alexandrov spaces.
	\end{proof}
		Now it is clear how to obtain theorem A from theorem B: Use theorem B to obtain a sequence of $(-2+\varepsilon_i)$-concave functions $f_i$ for a well chosen sequence $\varepsilon_i$. Multiplying these $f_i$ by a constant $(1+\delta_i)$ for appropriately chosen $\delta_i$, gives a sequence of $(-2)$-concave functions, which can be glued together to a $(-2)$-concave function using the same arguments as above, producing the map promised by theorem A.
\section{Product case}\label{sec: product case}
	The goal of this section is to prove the implication
		$$P(n,k+1) \Rightarrow P(n,k) \text{ for all } 1 \leq k \leq n-1.$$
	By \autoref{Cor: Boundary case} we can assume that $A$ has no boundary.\par
	Suppose the tangent space $T_pA$ at $p$ splits of an $\mathbb{R}^k$ factor, that is $T_pA$ is isometric to $ \mathbb{R}^k \times T_bB$ for some $b \in B$, where $B$ denotes an $(n-k)$-dimensional Alexandrov space.  Consider for sufficiently small $c>0$ the function
		$$f_c: \mathbb{R}^k \times T_bB \rightarrow \mathbb{R}; (x,y) \mapsto  - \frac{1}{2} \cdot \left(  \Vert x  \Vert^2 + \left( \vert o_{b} y \vert_{B} +c \right)^2  \right) .$$
	This function behaves well in the product situation. Meaning that in a small neighborhood of a point $q \in S_1(o_p)$ one can write down a Taylor approximation, which can be lifted to nearby Alexandrov spaces. For points $q \in S_1(o_p)$ not lying in $\mathbb{R}^k \times \lbrace o_b \rbrace$ one can define a map
		$${F}_q:= C + D_{q} + \frac{ (1-\varepsilon)}{2} \cdot f_{q}, $$
	where $C$ is a yet to be determined constant, $D_q$ is a weighted sum of Busemann functions imitating the gradient of $f_c$ and $f_{q}$ is a map coming from $P(n,k+1)$ mimicking the second order behavior of $f_c$. \par
	Obviously ${F}_q$ can be lifted to nearby Alexandrov spaces, since $f_q$ is constructible by $P(n,k+1)$ and Busemann functions
	can be approximated by distance functions of points lying sufficiently far away, which again have canonical lifts.\par
	Let us describe the intermediate goal: For every point $q \in B_{1+\varepsilon}(o_p) \setminus B_{\frac{1}{2}}(p)$ one would like to construct a $(-2+\varepsilon)$-concave and $(1+\varepsilon)$-Lipschitz function $F_q: B_{r_q}(q) \rightarrow \mathbb{R}$ satisfying
		\begin{align}\begin{split}\label{Equation:gluing condition}
			F_q(x) & \geq f_c(x)  \quad \forall x \in B_{r_q}(q) \setminus B_{\frac{2}{3}r_q}(q),\\
			F_q(x) & \leq f_c(x) \quad \forall x \in B_{\frac{1}{10}r_q}(q).		\end{split}
		\end{align}
	Cover a small compact neighborhood $N$ of $S_1(o_p)$ by finitely many $B_{\frac{r_1}{10}}(q_1),\ldots, B_{\frac{r_N}{10}}(q_N)$ and define $F:N \rightarrow \mathbb{R}$ by
		$$F(x) = \min F_i(x) \cdot \mathbbm{1}_i,$$
	where $\mathbbm{1}_i$ denotes the characteristic function for $B_{r_i}(q_i)$, i.e.
		$$ \mathbbm{1}_i(x):= \begin{cases} 1 & \text { if } x \in B_{r_i}(q_i) \\ 0 & \text { if } x \notin B_{r_i}(q_i)  \end{cases}.  $$
	If the $F_i$ would be globally defined, then $F$ would be clearly a $(-2+\varepsilon)$-concave and $(1+\varepsilon)$-Lipschitz function. Using the conditions above one arrives at a similar situation. Namely, fix an arbitrary point $y \in N$ and denote by $I_y$ the subset of $\lbrace 1, \ldots , N\rbrace$ such that $y$ lies in the interior of $B_{r_i}(q_i)$ for every $i \in I_y$. Then only the $i \in I_y$ are relevant for the definition of $F$ in a small neighborhood of $y$. Indeed for an index $j$ in the compliment of $I_y$ one can find $r_y>0$ small enough, such that $y$ lies in the interior of every $B_{r_i}(q_i)$ for $i \in I_y$ and $B_{r_y}(y) \cap B_{r_j}(q_j) \subset B_{r_j}(q_j) \setminus B_{\frac{2}{3}r_j}(q_j)$. One then has
		$$F(x) \leq f_c(x) \leq F_j(x) \quad \forall x \in B_{r_y}(y).$$
	In particular this shows
		$$ F\vert_{B_{r_y}(y)} = \min_{i \in I_y} F_i,$$
	and all $F_i$ are globally defined on $B_{r_y}(y)$, therefore $F\vert_{B_{r_y}(y)}$ is $(-2+\varepsilon)$-concave and $(1+\varepsilon)$-Lipschitz. Since Lipschitz-continuity and concavity are local properties the same is true for $F$. It will turn out in \autoref{subsec: convex region} that, if $F$ is in addition uniformly close to $-\dist_{o_p}^2$, the level set $\lbrace F=1 \rbrace$ bounds a convex region. \par
	The lifts of $F_q$ satisfy inequalities similar as in (\ref{Equation:gluing condition}), thus the arguments above can be repeated verbatim for the lifts. This implies that the function $F$ can be lifted and the level sets of the lift bound convex regions, these are precisely the sets $C_r$ mentioned in the introduction.
	\subsection{First order term}\label{subseq:first order term}
	Let us describe the map ${F}_q$ in more detail. The first step is to explicitly construct the first order term $D_q$. Fix a point $(v,w) \in \mathbb{R}^k \times T_bB$, if $v \neq 0$ there exists a ray $\gamma_1$ starting at $(0,o_{b})$ and going through $(v,o_{b})$. If $w \neq o_{b}$, then there exists a ray $\gamma_2$ starting at $(0,o_{b})$ and going through $(0,w)$. Associate to each ray a Busemann-function $B_1,B_2$ respectively. Straightforward computations in $T_pA$ show
		\begin{align*}
		B_1((x,y))&= -  \Vert x \Vert \cdot  \cos \left( \left   \vert \uparrow_0^x \uparrow_0^v  \right \vert \right), \\
		B_2((x,y))&=- \vert o_{b} y  \vert \cdot \cos \left( \left \vert \uparrow_{o_{b}}^y \uparrow_{o_{b}}^w  \right \vert \right ).
		\end{align*}
		In particular $B_1$ is constant on the second factor $T_bB$ and $B_2$ is constant on the first factor $\mathbb{R}^k$.
	Now define
		$$D_{(v,w)}: \mathbb{R}^k \times T_bB \rightarrow \mathbb{R} ; (x,y) \mapsto \Vert v \Vert B_1(x) + (\vert o_b w \vert + c ) B_2(y), $$
	which is the model gradient for $f_c$. This model gradient can be approximated using only distance functions. More precisely define for $t>0$ maps by
		$$ D^t_{(v,w)} ((x,y)) = \Vert v \Vert \cdot \dist_{\gamma_1(t)}((x,y)) + (\vert o_b w \vert + c ) \cdot \dist_{\gamma_2(t)}((x,y)),$$
	where $\gamma_1, \gamma_2$ are the rays mentioned above.\par
	We will exploit later on that $D^t_{(v,w)}$ is $(1+\varepsilon)$ Lipschitz in a neighborhood of $(v,w)$ if $(v,w)$ is in $S_1(o_p)$ and $c>0$ is sufficiently small. In order to prove this one needs a small lemma.
	\begin{lem}[Noncontracting map, \cite{MR1835418} Proposition 10.6.10, p.374] \label{Le: Noncontracting map}
		Let $A$ be an Alexandrov space of dimension $n$, curvature $\geq \kappa$ and let $p \in A$ be a point. Then there exists a map $f:A \rightarrow \mathbb{R}^n_{\kappa}$, such that $$\vert f(x) f(y) \vert \geq \vert xy \vert \text{ for all } x,y \in A$$ (i.e. $f$ is noncontracting) and $\vert f(p)f(x)\vert = \vert px \vert$ for all $x \in A$.
	\end{lem}
	\begin{lem}
		Let $q=(v,w) \in \mathbb{R}^k \times T_bB = T_pA$ be a point with $w \neq 0$ 
		Then one has that $D^t_{(v,w)}$ is $\sqrt{\Vert v \Vert^2  + (\vert o_b w \vert+ c)^2}$-Lipschitz for all $t>0$.
	\end{lem}
		\begin{proof}
			In order to prove the assertion for $D^{t}_{(v,w)}$ it is sufficient to bound their directional derivatives at all points $x \in B_R(q)$. For $x \in B_R(q)$ and $\uparrow \in \Sigma_{x}$ the directional derivative of $D^{t}_{(v,w)}$ at $x$ is given by
			\begin{align*}
			(D^{t}_{(v,w)})'_{x}(\uparrow)&= -\Vert v \Vert \cos\left(  \left \vert \Uparrow_{x}^{\gamma_1(t)} \uparrow \right \vert \right)  -(\vert o_b w \vert + c) \cos\left( \left \vert \Uparrow_{x}^{\gamma_2(t)} \uparrow \right \vert \right ) \\
			&= \left \langle \begin{pmatrix}
			\Vert v \Vert \\
			\vert o_b w \vert + c
			\end{pmatrix},
			\begin{pmatrix}
			-\cos\left( \left \vert \Uparrow_{x}^{\gamma_1(t)} \uparrow \right \vert \right)  \\
			-\cos\left( \left \vert \Uparrow_{x}^{\gamma_2(t)} \uparrow \right \vert \right)
			\end{pmatrix} \right \rangle .
			\end{align*}
			Using Cauchy-Schwartz inequality one gets
			$$(D^{t}_{(v,w)})'_{x}(\uparrow) \leq \sqrt{\Vert v \Vert ^2 + (\vert o_bw \vert +c)^2}  \cdot \sqrt{\sum_{i=1}^2 \cos^2\left( \left \vert \Uparrow_{x}^{\gamma_i(t)} \uparrow \right \vert \right)}. $$
			For $q=(v,w) \in \mathbb{R}^k \times T_bB \setminus \mathbb{R}^k \times \lbrace o_b \rbrace$ set $a_1:=\gamma_1(t) \in \mathbb{R}^k \times \lbrace w \rbrace$ for some sufficiently large $t>0$. Find points $a_2,\ldots, a_k, b_1, \ldots b_k \in \mathbb{R}^k \times \lbrace w \rbrace$ at distance $\vert a_1 q \vert$ from $q$ such that $\lbrace(a_i,b_i)\rbrace_{i=1}^k$ is a $(k,0)$-strainer at $q$.\par
		Since $q$ does not lie in the $\mathbb{R}^k$-factor, the ray $\gamma_2$ exists. Set $a_{k+1}:=\gamma_2(t)$ for some sufficiently large $t>0$ and find $b_{k+1}$ such that $\lbrace(a_i,b_i)\rbrace_{i=1}^{k+1}$ is a $(k+1,0)$-strainer at $q$. Again this is possible, since $q \notin \mathbb{R}^k \times \lbrace o_b \rbrace$.\par
		For every $\delta>0$ there exists $R>0$ such that $\lbrace (a_i,b_i)\rbrace_{i=1}^{k+1}$ is a $(k+1,\delta)$-strainer for all points in $B_R(q)$. The constant $\delta$ is assumed to be small and will be specified below. Applying \autoref{Le: Noncontracting map} for $x \in B_R(q)$ and $\uparrow \in \Sigma_x$ there exists a noncontracting map $E: \Sigma_{x} \rightarrow S^{n-1}$ satisfying $\vert E(\uparrow) E(\uparrow')\vert=\vert \uparrow \uparrow' \vert $ for all $\uparrow' \in \Sigma_x$.\par
	The $(k+1,\delta)$-strainer $\lbrace(a_i,b_i)\rbrace_{i=1}^{k+1}$ induces an $(k+1,\delta)$-explosion $(\Uparrow_x^{a_i},\Uparrow_x^{b_i})_{i=1}^{k+1}$ in $\Sigma_x$ in the sense of \autoref{Def: (n,d) explosion}. Since $E$ is noncontracting, it maps this $(k+1,\delta)$-explosion in $\Sigma_x$ to a $(k+1,\delta)$-explosion in $S^{n-1}$. Denote this explosion by $(A_i,B_i)_{i=1}^{k+1}$.\par

	In $S^{n-1}$ there is a canonical way to extend the $(k+1,\delta)$-explosion to an $(n,\delta)$-explosion $(A_i,B_i)_{i=1}^n$. Then \autoref{Lemma:BGP} implies
\begin{align*}
\sum_{i=1}^2 \cos^2\left( \left \vert \Uparrow_{x}^{\gamma_i(t)} \uparrow \right \vert \right) & = \sum_{i=1}^2 \cos^2\left( \left \vert A_i E(\uparrow) \right \vert \right)
 \leq \sum_{i=1}^{n} \cos^2\left( \left \vert A_i E(\uparrow) \right \vert \right) \leq 1 + \varepsilon(\delta),
\end{align*}
which proves the result.
		\end{proof}
\subsection{Second order approximation}\label{subseq second order approx}
The tangent space at the point 
	$$(v,w) = q \in B_{1+\varepsilon}(o_p)\setminus \mathbb{R}^k \times \lbrace o_b \rbrace$$ satisfies $T_q(T_pA)=\mathbb{R}^{k+1} \times T_{\overline{b}}\overline{B}$ by the splitting theorem \cite{MR0256327}, since it is lying in the interior of $k+1$ pairwise orthogonal shortest paths. Therefore, one can use the statement $P(n,k+1)$ to find $r_q>0$ and a function $f_{\delta}:B_{r_q}((v,w))\rightarrow \mathbb{R}$ for $\delta:= \varepsilon/1000$ as in \autoref{Main Proposition}. Consider the function ${F}_q: B_{r_q}((v,w))\rightarrow \mathbb{R}$ defined by the formula
$$ {F}_q((x,y)):= C + D_{(v,w)}((x,y)) + \frac{ (1-\varepsilon)}{2} \cdot f_{\delta}((x,y)),$$
where $D_{(v,w)}$ is defined in \autoref{subseq:first order term} and the constant $C$ is given by
$$ C:= \frac{\Vert v \Vert^2 +(\vert o_b w \vert + c ) \vert o_bw \vert - c^2}{2} - \frac{ r_q^2 \varepsilon}{9} .$$
One has for all $(x,y) \in B_{r_q}((v,w)) \setminus B_{\frac{2}{3} r_q}((v,w))$
\begin{equation*}\label{Eq: Second order approximation inequality}
{F}_q((x,y)) \geq f_c((x,y)) + \frac{ r_q^2 \varepsilon}{9}.
\end{equation*}
Indeed observe the identities
$$ B_1(x) - \frac{\Vert x - v \Vert^2}{2 \Vert v \Vert} +  \frac{\Vert v \vert^2}{2\Vert v \Vert} = - \frac{\Vert x \Vert^2}{2 \Vert v \Vert}, B_2(y) - \frac{\vert wy \vert^2  }{2 \vert o_b w \vert}+\frac{\vert o_b w \vert^2 }{2 \vert o_b w \vert}=-\frac{\vert o_b y \vert^2}{2 \vert o_b w \vert}. $$
By construction $f_{\delta}((x,y)) \geq -\vert (x,y) (v,w)\vert ^2 =- \Vert x-v \Vert^2 - \vert wy \vert^2  $ and thus with the above
\begin{align*}
{F}_q((x,y)) & \geq C + \Vert v \Vert B_1(x) + (\vert o_b w \vert + c ) B_2(y) - \frac{ \Vert x-v \Vert^2 + \vert wy \vert^2 }{2} \\
& \phantom{\geq} + \frac{\varepsilon \cdot  \vert (x,y) (v,w) \vert^2}{2} \\
& = \Vert v \Vert  \left(B_1(x) - \frac{\Vert x - v \Vert^2}{2 \Vert v \Vert} +  \frac{\Vert v \Vert^2}{2\Vert v \Vert}  \right)  \\
& \phantom{=} + (\vert o_b w \vert +c) \left( B_2(y) - \frac{\vert wy \vert^2  }{2 \vert o_b w \vert}+\frac{\vert o_b w \vert^2 }{2 \vert o_b w \vert} \right) + \frac{c \vert wy \vert^2}{2 \vert o_b w \vert} \\
&\phantom{=}- \frac{c^2}{2} + \frac{\varepsilon \cdot  \vert (x,y) (v,w) \vert^2}{2}- \frac{2 r_q^2 \varepsilon}{9}\\
& \geq - \frac{\vert (x,y) (0,o_b) \vert^2}{2} - \frac{c \vert o_b y \vert^2}{ 2\vert o_b w \vert}-\frac{c^2}{2} +  \frac{\varepsilon \cdot  \vert (x,y) (v,w) \vert^2}{2}- \frac{ r_q^2 \varepsilon}{9}.
\end{align*}
If $y$ is sufficiently close to $w$, one has $ -{c \vert o_b y \vert^2} /2{\vert o_b w \vert} \geq - c \vert o_b y \vert $, therefore if $r_q$ is sufficiently small, using the above one obtains
\begin{equation}\label{Eq: Preequation for proving result in Proposition in product case}
{F}_q((x,y)) \geq f_c((x,y)) +  \frac{\varepsilon \cdot  \vert (x,y) (v,w) \vert^2}{2}- \frac{ r_q^2 \varepsilon}{9}
\end{equation}
for all $(x,y) \in B_{r_q}((v,w)) \setminus B_{\frac{2 r_q}{3}}((v,w))$.\par

Now observe the following: \begin{enumerate}
	\item The arguments in \autoref{subseq:first order term} carry over verbatim to the lifts of $F_q$, which are defined in a canonical way, thus $F_q$ and its lifts are $(1+\varepsilon)$-Lipschitz functions.
	\item  Moreover $F_q$ and its lifts will be $(-2+\varepsilon)$ concave by construction. 
	\item Clearly $F_q$ is uniformly close to $f_c$ and thus to $-\dist_{o_p}^2$ if $r_q$ is sufficiently small, additionally one has $F_q < f_c$ in a small neighborhood around $q$ and the same is true for close lifts, since $F(q)<f_c(q)$.
Therefore we are in the situation described at the paragraphs preceding \autoref{subseq:first order term}. It remains to show that a similar situation can be achieved at points in the $\mathbb{R}^k$-factor.
\end{enumerate} 
\subsection{Second order approximation in the $\mathbb{R}^k$-factor}\label{subseq: second order with factor}
	For points $q \in \mathbb{R}^k \times \lbrace o_b \rbrace$ one cannot apply $P(n,k+1)$ to get a second order approximation function $f_{\delta}$ as in \autoref{subseq second order approx}. Now the reason for the introduction of the functions $f_c$ becomes clear:\par At points in the $\mathbb{R}^k$ factor it is possible to use weaker functions (in the sense that they do not satisfy the lower bound given in theorem B) to approximate $f_c$ up to second order such that (\ref{Equation:gluing condition}) is satisfied. More precisely one needs:
	\begin{lem}[Weak second order approximation]\label{Le: Weak second order approximation with lift}
	For every $q = (v,o_b) \in \mathbb{R}^k \times \lbrace o_b \rbrace \subset T_pA$ and every $\varepsilon>0$ there exists $R>0$ and a map $H:B_R(q) \subset T_pA \rightarrow \mathbb{R}$ satisfying the following list of conditions:
		\begin{enumerate}
		\item The functions $H$ are $\varepsilon$-Lipschitz and $(-2+\varepsilon)$-concave on their domain of definition.
		\item The function $H$ satisfies $H(q)=0$ and $\nabla_{q}H=o_q$.
		\item The function $H$ satisfies for all $(x,y)\in B_R((v,o_b)) \subset \mathbb{R}^k \times T_bB$
		\begin{equation}\label{Eq: Lower bound for weak second order approximation}
		H((x,y)) \geq - {\Vert x-v \Vert^2} - \varepsilon \cdot \vert o_b y \vert - \varepsilon^2 \cdot \vert (x,y) (v,o_b) \vert^2.
		\end{equation}
	\end{enumerate}
\end{lem}
	How to obtain these weaker functions? On the $\mathbb{R}^k$-factor one can use the construction coming from \autoref{subseq Model constrution}. On the $T_bB$-factor the construction of Kapovitch given in \cite{MR1904560} will be used. Both constructions yield functions defined in terms of distance functions of a finite number of points, in particular the definition of the maps $H_1 :\mathbb{R}^k \rightarrow \mathbb{R}$, and $ H_2:T_bB \rightarrow \mathbb{R}$ a priori defined on each factor separately can be canonically extended to the product $\mathbb{R}^k \times T_bB$.\par
	We want to make precise what the functions $H_1,H_2$ look like, we call the map $H_1$ the flat-factor term and $H_2$ the cone-factor term.
\begin{dfn}[Flat-factor term]
Denote by $\lbrace(p_i^+,p_i^-)\rbrace_{i=1}^k$ a $(k,0)$-strainer around $0 \in \mathbb{R}^k$, where all points have distance one to $0 \in \mathbb{R}^k$. Then $\lbrace (p_i^+,o_b),(p_i^-,o_b) \rbrace_{i=1}^k$ is a $(k,0)$-strainer around $o_p \in T_pA$. Consider for $R>0$ the real valued function
$$\varphi_{R}: \mathbb{R} \rightarrow \mathbb{R}; z \mapsto  \frac{-(1+R-z)^2 + R}{2}$$
and define the map $H_1^{R}:\mathbb{R}^k \times T_bB \rightarrow \mathbb{R}$ by
\begin{equation}
H_{1}^{R}((x,y))=  \sum_{i=1}^{k} \varphi_R \circ \dist_{(p_i^+,o_b)}((x,y)) + \varphi_R \circ \dist_{(p^-_i,o_b)}((x,y)).
\end{equation}
The map $H_1^{R}$ is called the flat-factor term of the second order approximation. To unburden notation the index $R$ will be omitted.
\end{dfn}

	\begin{dfn}[Cone-factor term]\label{Def: Cone factor term}
	Write the tangent space $T_bB$ as $T_bB = \operatorname{Cone}(\Sigma_{b})$. Fix an $0.1$-net $\lbrace q_1 , \ldots, q_{N} \rbrace$ in $\Sigma_b$. For each $ i \in \lbrace 1, \ldots , N \rbrace$ consider the set $B_{0.1}(q_i) \subset \Sigma_b$ and find for $\delta>0$ a maximal $\delta$-separated set $\lbrace q^i_{1}, \ldots , q^i_{N_i} \rbrace \subset B_{0.1}(q_i) \subset \Sigma_b$ (i.e. $\vert q_k^i q _l^i \vert \geq \delta$ for all $1 \leq k \neq l \leq N_i$). Define for $K>0$ the map
	$$ H_{2,i}^{R,K,\delta}: B_R(o_p) \subset \mathbb{R}^k \times T_bB \rightarrow \mathbb{R}; (x,y) \mapsto \frac{K}{N_i} \cdot \sum_{\alpha=1}^{N_i} \varphi_{R} \circ \dist_{(0,q^i_{\alpha})}((x,y)) . $$
	To unburden notation the indices $R,\delta ,K$ are omitted. Define $H_{2}: B_R(o_p) \subset \mathbb{R}^k \times T_bB \rightarrow \mathbb{R}$ by
	\begin{equation}
	H_{2}((x,y))= \min_{1 \leq i \leq N} H_{2,i}((x,y)).
	\end{equation}
	The map $H_{2}$ is called the cone-factor term of the second order approximation with parameters $R,K,\delta$.
\end{dfn}

Lets start with the proof of \autoref{Le: Weak second order approximation with lift}.

\begin{proof}[Proof of \autoref{Le: Weak second order approximation with lift}]
	It is sufficient to prove that $H_1 + H_{2,i}$ is $(-2+\varepsilon)$-concave and $\varepsilon$-Lipschitz on $B_R(o_p)$ (for the definition of $H_{2,i}$ see \autoref{Def: Cone factor term}). \par 
	All terms in the definition of $H_1$ and $H_{2,i}$ are $\varepsilon$-Lipschitz, if the parameters are chosen appropriately, therefore $H$ is also $\varepsilon$-Lipschitz.\par
	In order to prove $(-2+\varepsilon)$-concavity of $H_i:=H_1 + H_{2,i}$ similar to the proof of \autoref{Lem: Concavity of lifts} it is sufficient to show
	$$ 2 H_i(m) - H_i(x) - H_i(y) \geq   \left(1-\frac{\varepsilon}{2} \right) \frac{\vert xy \vert ^2 }{2} $$
	for any $x,y \in B_R(o_p)$ and every midpoint $m$ between $x$ and $y$. Recall from the proof of \autoref{Lem: Concavity of lifts}: If $R>0$ is sufficiently small, then the above expression $2 H_i(m) - H_i(x) - H_i(y)$ is bounded below by (up to terms $\varepsilon  \vert xy \vert^2$)
	\begin{align*}
	\sum_{i=1}^k \left[ \cos^2(\vert \uparrow_m^x \uparrow_m^{(p_i^+,o_b)} \vert) +\cos^2(\vert \uparrow_m^y \uparrow_m^{(p_i^+,o_b)} \vert)  \right] \frac{\vert xy \vert^2}{8}  \\
	+\sum_{i=1}^k \left[ \cos^2(\vert \uparrow_m^x \uparrow_m^{(p_i^-,o_b)} \vert) +\cos^2(\vert \uparrow_m^y \uparrow_m^{(p_i^-,o_b)} \vert)  \right] \frac{\vert xy \vert^2}{8} \\
	+ \frac{K}{N_i}\sum_{i=1}^{N_i} \left[ \cos^2( \vert \uparrow_m^x \uparrow_m^{(0,q_i)} \vert ) + \cos^2( \vert \uparrow_m^y \uparrow_m^{(0,q_i)} \vert )  \right] \frac{\vert xy \vert^2}{8}.
	\end{align*}
	The first four terms are similar in their nature. Namely one computes the distance of some direction in the space of directions to an $(k,\delta')$-explosion (see \autoref{Def: (n,d) explosion}), where $\delta' \rightarrow 0 $ if $R \rightarrow 0$. Thus it is sufficient to show for arbitrary $\uparrow \in \Sigma_m$ the inequality
	\begin{align}\label{Eq: Product concavtiy lemma}
	\sum_{i=1}^k \cos^2(\vert \uparrow \uparrow_m^{(p_i^+,o_b)} \vert)+ \frac{K}{2 N_i}\sum_{i=1}^{N_i}  \cos^2( \vert \uparrow \uparrow_m^{(0,q_i)} \vert ) \geq \left(1-\frac{\varepsilon}{4} \right) .
	\end{align}
	Observe that the statement is satisfied in the model situation, that is if $T_pA= \mathbb{R}^k \times \mathbb{R}^{n-k}$, then the tangent space of $T_pA$ is given by the spherical join $S^{k-1} \ast S^{n-k-1}$. If $\uparrow$ lies in the $S^{k-1}$-factor the inequality is true by the computations made in \autoref{sec: reg case} for the first term, by continuity the same remains true for all $\uparrow$ lying in a small neighborhood around $S^{k-1}$.\par
	Analogously if $\uparrow$ lies in $S^{n-k-1}$ the inequality is also true, since the second term satisfies a much stronger inequality. More precisely, for dimensional reasons most of the summands of the second term will be bigger than some fixed bound, by choosing the parameters appropriately the second term can be made as large, as one wants it to be. In particular one can achieve that the second term satisfies the inequality for all $\uparrow$ lying outside a small neighborhood of $S^{k-1}$ (a more detailed description of this argument is given in \cite{Nepechiy2018} and \cite{MR1904560}).\par
	The general case follows from \autoref{Le: Noncontracting map}. Recall that the space of directions of the tangent space of $T_pA= \mathbb{R}^k \times T_bB$ is given by the spherical join $S^{k-1} \ast \Sigma_{b}$, where $\Sigma_{b}$ denotes the space of directions of $T_bB$. For $\uparrow$ one can construct a noncontracting map $f: S^{k-1} \ast \Sigma_{b} \rightarrow S^{k-1} \ast S^{n-k-1}$, such that all distances to $\uparrow$ are preserved. Thus one is in the model situation. \par
	Since the map is noncontracting, it maps points in $S^{k-1}$ to points in $S^{k-1}$ and points in $\Sigma_b$ to points in $S^{n-k-1}$. Moreover an $(k,\delta)$-strainer in $S^{k-1}$ will be mapped to a $(k,\delta)$-strainer. The dimensional argument mentioned for the model situation is also true for $\Sigma_{b}$ and by the noncontracting properties carries over to the model situation. This proves concavity of $H_i$ and thus also concavity of $H$. \par
	The validity of (\ref{Eq: Lower bound for weak second order approximation}) can be seen the following way: Consider instead of $H_1$ the map $H_1 \circ \pi_1$, where $\pi_1$ is the projection on the first factor
	$$\pi_1 : \mathbb{R}^k \times T_bB \rightarrow \mathbb{R}^k \times \lbrace o_b \rbrace; (x,y) \mapsto (x,o_b)$$
	and instead of $H_{2,i}$ the map $H_{2,i} \circ \pi_2$, where $\pi_2$ is the projection on the second factor
	$$\pi_2: \mathbb{R}^k \times T_bB \rightarrow \lbrace 0 \rbrace \times T_bB; (x,y) \mapsto (0,y).$$
	Using the same computations as in \autoref{sec: reg case} one gets $H_1 \circ \pi_1((x,y)) \geq - \Vert x \Vert^2$.\par
	Set $H_2 \circ \pi_2 := \min_i H_{2,i} \circ \pi_2$, then $H_2 \circ \pi_2$ restricted to $T_bB$ is obviously $\varepsilon$-Lipschitz as the composition of $\varepsilon$-Lipschitz maps, thus together with $H_2 \circ \pi_2((0,o_b))=0$ one has $H_2\circ \pi_2 ((x,y)) \geq - \varepsilon  \vert o_b y \vert$. Observe by making $R$ sufficiently small one has
	$$\vert H((x,y)) - (H_1 \circ \pi_1 ((x,y)) +H_2 \circ \pi_2 ((x,y))) \vert \leq \varepsilon^2 \cdot \vert (x,y) (0,o_b)\vert^2 .$$
	Indeed this just follows from the definition of product metric and straightforward computations.
	Combining these estimates shows (\ref{Eq: Lower bound for weak second order approximation}) and finishes the proof.
\end{proof}

	Observe that the arguments used in the proof of \autoref{Le: Weak second order approximation with lift} carry over almost verbatim to the canonically defined lifts. Thus we arrive at the situation described at the beginning of \autoref{sec: product case}. To summarize what was achieved a new definition is necessary. 
	\begin{dfn}[Pseudo-constructible]\label{Def: Pseudo-constructible}
		Fix $\varepsilon>0$ and let $p \in A$ be a point in an $n$-dimensional Alexandrov space. Denote by $T_pA$ the tangent space at $p$ and by $o_p$ its apex. Assume
		$$F_{\varepsilon}:B_{1+\varepsilon}(o_p) \setminus \overline{B}_{\frac{1}{2}}(o_p) \subset T_pA \rightarrow \mathbb{R}$$
		is a $(1+\varepsilon)$-Lipschitz and $(-1+\varepsilon)$-concave function satisfying
		$$ -(1-\varepsilon^2) \frac{\vert x o_p \vert^2}{2}  \geq  F_{\varepsilon}(x) \geq - (1+\varepsilon^2) \frac{\vert x o_p \vert^2}{2} \text { for all } x \in B_{1+\varepsilon}(o_p) \setminus \overline{B}_{\frac{1}{2}}(o_p). $$
		The function $F_{\varepsilon}$ is called pseudo-constructible if for any pointed sequence $(A_i^n,p_i)\rightarrow (T_pA,o_p)$ of $n$-dimensional Alexandrov spaces $A_i^n \in \operatorname{Alex}(\kappa_i)$ for some sequence $\kappa_i \rightarrow 0$ for $i \rightarrow \infty$, there exists $N \in \mathbb{N}$ such that for all $i \geq N$ there is a $(1+\varepsilon)$-Lipschitz, $(-1+\varepsilon)$-concave function
		
		$$F_{\varepsilon,i}:B_{1+\varepsilon}(p_i) \setminus \overline{B}_{\frac{1}{2}}(p_i) \subset A^n_i \rightarrow \mathbb{R},$$
		which satisfies
		$$ -(1-2\varepsilon^2) \frac{\vert x p_i \vert^2}{2}  \geq  F_{\varepsilon,i}(x) \geq - (1+2\varepsilon^2)\frac{\vert x p_i \vert^2}{2}  \text { for all } x \in B_{1+\varepsilon}(p_i) \setminus \overline{B}_{\frac{1}{2}}(p_i).$$
	\end{dfn}

Combining everything from \autoref{sec: product case} so far yields \autoref{Prop: Construction of defining function}. The idea is: We have obtained a pseudo-constructible function, from there one can with some effort obtain a constructible one in the sense of \autoref{Definition: Constructible}. The last part of \autoref{sec: product case} will explain how to achieve exactly that. \par

\begin{lem}[Construction of convex region function]\label{Prop: Construction of defining function}
	Let $A$ be an $n$-dimensional Alexandrov space without boundary and let $p \in A$ be a point. Fix an arbitrary $\varepsilon>0$. Denote by $T_pA$ the tangent space at $p$.\par
	Then for every $\varepsilon >0$ there exists a function
	$$F_{\varepsilon}:B_{1+\varepsilon}(o_p) \setminus \overline{B}_{\frac{1}{2}}(o_p) \subset T_pA \rightarrow \mathbb{R}$$
	satisfying the following conditions:
	\begin{enumerate}
		\item \label{Property1: in pseudo constructible proposition} The function $F_{\varepsilon}$ is $(1+\varepsilon)$-Lipschitz and $(-1+\varepsilon)$-concave on its domain of definition.
		\item \label{Property2: in pseudo constructible proposition} The function $F_{\varepsilon}$ satisfies
		$$ -(1-\varepsilon^3) \frac{\vert x o_p \vert^2}{2}  \geq  F_{\varepsilon}(x) \geq - (1+\varepsilon^3)\frac{\vert x o_p \vert^2}{2}  \text { for } x \in B_{1+\varepsilon}(o_p) \setminus \overline{B}_{\frac{1}{2}}(o_p). $$
		In particular this implies for the Hausdorff-distance $d_H$
		$$d_H\left(S_1(o_p),\lbrace F_{\varepsilon}(x) = -1/2  \rbrace \right) \leq 2 \varepsilon^3. $$
		\item The function $F_{\varepsilon}$ is pseudo-constructible in the sense of \autoref{Def: Pseudo-constructible}.
	\end{enumerate}
\end{lem}
	\subsection{The convex region}\label{subsec: convex region}
	Consider the function $F_{\varepsilon}$ coming from \autoref{Prop: Construction of defining function}. Notice that the level set $\lbrace F_{\varepsilon} = -1/2 \rbrace$ bounds a convex region. \par
	Indeed the set $C_{\varepsilon}:=\lbrace F_{\varepsilon} \geq -1/2 \rbrace \cup B_{\frac{1}{2}}(o_p)$ by construction of $F_{\varepsilon}$ is Hausdorff close to $B_1(o_p)$. For arbitrary $x,y \in C_{\varepsilon}$ one has: If $d(x,y) \leq \varepsilon/4$, then the whole shortest path between $x$ and $y$ can not leave $C_{\varepsilon}$. Either it is contained in the domain of definition of $F_{\varepsilon}$, then the claim follows by concavity or one of the points lies in $B_{1/2}(o_p)$ but then the shortest path cannot leave for example $B_{3/4}(o_p)$ and is again in $C_{\varepsilon}$.\par
	Now assume $d(x,y) \geq  2 \cdot  \varepsilon/4$. If $m$ is a midpoint between $x$ and $y$ the distance in the model  $d( m, o_p)$ can be computed from the Euclidean situation, since $T_pA$ is an Euclidean cone. In particular one obtains $  d( m, o_p) \leq 1-2\varepsilon^3$ and therefore $m \in C_{\varepsilon}$. By induction it follows that for every $x,y, \in C_{\varepsilon}$ all shortest paths between $x$ and $y$ are contained in $C_{\varepsilon}$.\par
	Observe that this argument is, with the obvious modifications, also applicable for the lifts of $F_{\varepsilon}$. Now the Set $C_{\varepsilon}$ is convex with nonempty Alexandrov boundary. The next step is to show that the distance function from $\partial C_{\varepsilon}$ is more concave than the distance function from the sphere in the Euclidean situation.\par
	In \cite{MR2408266}[Thm. 3.3.1] it was proven that $\dist_{\partial A}$ is a concave function given that $\partial A \neq \emptyset$ and $A$ has curvature $\geq 0$. The idea of this proof is to compare $\dist_{\partial A}$ along a geodesic with a suitable comparison situation. Our problem is very similar in nature, one has additional assumptions on the boundary and wants to obtain a stronger concavity result for $\dist_{\partial C_{\varepsilon}}$. The way to do it is to construct a more adapted comparison situation, it will be described in \autoref{Def: Comparison for convex sets}. After that the proof for the concavity estimates will be carried out in \autoref{Prop: concavity estimates}.
\begin{dfn}[Model halfspace]\label{def: model halfspace}
	Denote by $\mathbb{R}^2_{\kappa}$ the $\kappa$-plane, i.e. the two-dimensional simply connected space form with constant curvature $\kappa$. Denote by $\mathbb{R}^{2+}_{\kappa}$ the model halfspace of the $\kappa$-plane that is the upper hemisphere in $S^2_{\kappa}$, the upper half-plane in $\mathbb{R}^2$ and the right quadrant in the upper half space model of the hyperbolic space $\mathbb{H}^2_{\kappa}$.
\end{dfn}

\begin{dfn}[Comparison for convex sets]\label{Def: Comparison for convex sets}
	Let $C$ be an Alexandrov space of dimension $n$ with curvature $\geq \kappa$, $\partial C \neq \emptyset$ and $\gamma$ a unit-speed shortest path with $\gamma(0) \in C \setminus \partial C$. Denote by $p \in \partial C$ the nearest point to $\gamma(0)$. Let $\alpha$ be the angle between $\gamma$ and the unique shortest path from $\gamma(0)$ to $p$, i.e. $\alpha:= \sphericalangle(\uparrow_{\gamma(0)}^{\gamma(t)}, \uparrow_{\gamma(0)}^p)$.\par
	Fix $\varepsilon>0$ and assume $\dist_{\partial C}(\gamma(0))<(1+\varepsilon)/(1-\varepsilon)$. Construct a comparison situation in the model halfspace: \par
	Choose $\overline{\gamma}(0) \in \mathbb{R}^{2+}_{\kappa} \setminus \partial \mathbb{R}^{2+}_{\kappa}$ satisfying $\vert \overline{\gamma}(0) \partial \mathbb{R}^{2+}_{\kappa} \vert= \vert p \gamma(0) \vert$, denote by $\overline{p}$ the closest point on $\partial \mathbb{R}^{2+}_{\kappa}$ to $\overline{\gamma}(0)$. Fix a unit-speed shortest path $\overline{\gamma}$ starting at $\overline{\gamma}(0)$ satisfying $\sphericalangle(\uparrow_{\overline{\gamma}(0)}^{\overline{p}}, \uparrow_{\overline{\gamma}(0)}^{\overline{\gamma}(t)})=\alpha$. Find the point $\overline{M}=\overline{M}(\varepsilon)$ with $\vert \overline{M} \overline{p} \vert = (1+\varepsilon)/(1-\varepsilon)=R(\varepsilon)=: R$, such that $\overline{\gamma}(0)$ lies on the shortest path between $\overline{M}$ and $\overline{p}$.\par
	Such a configuration is called $\varepsilon$-comparison for the convex set $C$ and the unit speed shortest path $\gamma$.
\end{dfn}

\begin{prp}[Concavity estimates]\label{Prop: concavity estimates}
	Let $A$ be an Alexandrov space of dimension $n$, with lower curvature bound $\geq \kappa$ and without boundary. Fix $\varepsilon>0$. Let $C \subset A$ be convex, compact, with $\partial C \neq \emptyset$ and $\mathring{C} \neq \emptyset$. Assume moreover that for each $b \in \partial C$ there exits an $r>0$ and a function $F_b : B_{r}(b) \rightarrow \mathbb{R}$ such that
	\begin{enumerate}
		\item \label{Property 1 from Concavity estimate Proposition} The function $F_b:B_{r}(b) \rightarrow \mathbb{R}$ is $(1+\varepsilon)$-Lipschitz and $(-1+\varepsilon)$-concave on $B_{r}(b)$.
		\item \label{Property 2 from Concavity estimate Proposition} One has $F_b(q) = -1/2$ if $q \in B_{r}(b) \cap \partial C_{} $ and $F_b(q) > -1/2 $ if $q \in B_{r}(b) \cap \mathring{C}$.
		\item \label{Property 3 from Concavity estimate Proposition}
		One has $\dist_{\partial C}(x) < R= R(\varepsilon):=\frac{1+\varepsilon}{1-\varepsilon}$ for all $x \in C$.
	\end{enumerate}
	By (\ref{Property 3 from Concavity estimate Proposition}) for a unit-speed shortest path $\gamma$ satisfying $\gamma(0) \in C \setminus \partial C$ the $\varepsilon$-comparison for the convex set $C$ as in \autoref{Def: Comparison for convex sets} is defined.\par
	Then the following inequality holds
	$$\dist(\partial C_{}, \gamma(t) )  \leq  \begin{cases} \dist(\partial B_{R}(\overline{M}), \overline{\gamma}(t)) +o(t^2) &, \kappa\geq 0 \\  \dist(\partial B_{R}(\overline{M}), \overline{\gamma}(t)) +o(t^2)
	+ T(\varepsilon,\kappa)  \cdot O(t^2) &,\kappa<0
	\end{cases}, $$
	where $T(\varepsilon,\kappa)$ is given by
	$$ T(\varepsilon, \kappa):= \sqrt{\vert \kappa \vert} \cdot \coth(\sqrt{ \vert \kappa \vert  R(\varepsilon)})-\frac{1}{R(\varepsilon)}. $$
	In particular if $\kappa \geq 0$, the distance function from $\partial C_{}$ is more concave than the distance function from $\partial B_{R}(\overline{M})$ in the Euclidean plane.\par
\end{prp}
	\begin{proof}
		As in \cite{MR2408266}[Thm. 3.3.1] the tangent space at $p$ splits as $T_pA = \mathbb{R}^+ \times T_p \partial C$. Therefore for the direction $\uparrow$ of an arbitrary shortest path from $p$ to $\gamma(t)$ one can write $\uparrow =(s,v) \in \mathbb{R}^+ \times T_p$. Consider the quasi-geodesic $q(t)$ in $A$  starting at $p$ and going in the direction $(0,v/\Vert v \Vert)$ (it exists by \autoref{Thm: Existence of Quasigeodesics}). Petrunin obtained a bound for the angle $\sphericalangle (\uparrow_p^{\gamma(t)},(0,v))$ in terms of the angle in the comparison situation, that is $\sphericalangle (\uparrow_{\overline{p}}^{\overline{\gamma}(t)},  \uparrow_{\overline{p}}^{\overline{q}})$, where $q$ denotes the direction of a geodesic $g(t)$ in the comparison situation starting in $\overline{p}$ and staying in $\partial \mathbb{R}^{2+}_{\kappa}$. More precisely one has
			$$ \sphericalangle ( \uparrow_{p}^{\gamma(t)}, (0,v)) \leq  \sphericalangle( \uparrow_{\overline{p}}^{\overline{\gamma}(t)}, \uparrow_{\overline{p}}^{\overline{q}})+ o(t). $$
		Using the angle monotonicity property of quasigeodesics one obtains for every $\tau>0$ an upper bound on $\vert q(t) \gamma(\tau) \vert$, which is of the form
			$$  \vert q(t) \gamma(\tau) \vert \leq \vert g(t) \overline{\gamma}(\tau) \vert + o(\tau) \cdot O(t).$$
		
		In particular one can choose $\lambda(t, \alpha, R, \kappa)$ such that the point $\overline{\gamma}(t)$ lies on a shortest path between $\overline{M}$ and $g(\lambda \cdot t)$. The above then implies the first of the two key estimates
			$$  \vert q( \lambda \cdot t) \gamma(t) \vert \leq \vert g( \lambda \cdot t) \overline{\gamma}(t) \vert + o(t^2).$$
		Since $F_b$ is $(-1+\varepsilon)$-concave and the directional derivative of $F_b$ in the direction of $(0,v)$ vanishes one has by (\ref{Property 2 from Concavity estimate Proposition}) of $F_b$ that $q(\lambda t)$ lies outside of $C$. Thus there exists a point in the boundary of $C$ denoted by $b(\lambda t)$ (in particular $F(b(\lambda(t)))=F(p)$), which lies on a shortest path between $q(\lambda \cdot t)$ and $ \gamma(t)$. Since $F_b$ is $(1+\varepsilon)$-Lipschitz, one has
			$$\frac{1-\varepsilon}{2}\cdot \lambda^2 \cdot  t^2 \leq  F(b(\lambda t)) - F(q( \lambda t)) \leq (1+\varepsilon) \cdot   \vert b(\lambda t) q(\lambda t) \vert .$$
		With this one can estimate $\dist_{\partial C}(\gamma(t))$ from above.
				\begin{align}\label{Eq: Concavity estimate for convex set last equation}
\begin{split}
\vert \gamma(t) \partial C \vert  & \leq \vert b(\lambda t) \gamma(t) \vert 
 = \vert q(\lambda t) \gamma(t) \vert - \vert b(\lambda t) q(\lambda t) \vert  \\
& \leq \vert g(\lambda t) \overline{\gamma}(t) \vert -\frac{1}{2R}\cdot \lambda^2  t^2  + o(t^2) 
= \vert \overline{M} g(\lambda t) \vert - \vert \overline{M} \overline{\gamma}(t) \vert -\frac{1}{2R}\cdot \lambda^2  t^2  + o(t^2) \\
\end{split}
\end{align}
		Using $\vert \overline{M} \overline{\gamma}(t) \vert = R - \dist(\partial B_R(\overline{M}), \overline{\gamma}(t))$ and the Taylor approximation of $\dist_{\overline{M}}$ at $p$ one obtains
				$$ \dist(\partial B_R(\overline{M}), \overline{\gamma}(t)) + \left(\operatorname{Hess}^{\kappa,\perp}_{\overline{p}}(\dist_{\overline{M}}(x)) - \frac{1}{R} \right) \cdot   \frac{ \lambda^2 t^2}{2}  + o(t^2),  $$
		where $ \operatorname{Hess}^{\kappa,\perp}_{\overline{p}}(\dist_{\overline{M}}(x))$ denotes the Hessian of $f(x)= \dist_{\overline{M}}(x)$ in the comparison space $\mathbb{R}_k^2$ at the point $\overline{p}$ in a direction perpendicular to $\uparrow_{\overline{p}}^{\overline{\gamma}(0)}$. This finishes the proof.
	\end{proof}
	Using \autoref{Prop: concavity estimates} one can immediately conclude $P(n,k)$. For this just consider a function if the form
		$$f_{C_{\delta}}:  \mathring{C}_{\delta} \rightarrow \mathbb{R}; x \mapsto -  (R(\delta) - \vert x \partial C_{\delta} \vert )^2. $$
	From \autoref{Prop: concavity estimates} one can, for given $\varepsilon>0$, deduce $(-2+\varepsilon)$ concavity of $f_{C_{\delta}}$, for sufficiently small $\delta>0$. Together with \autoref{Prop: Construction of defining function} one arrives at the situation of \autoref{Corollary: Self-improvement assumption}. From here one finishes the proof by repeating the self-improvement procedure described in \autoref{subsec:Improvement}.
\subsection{Final step}\label{sec: final step}
		In the final step one can assume that $P(n,1)$ is true. At any point $q \in T_pA \setminus \lbrace o_p \rbrace$ the tangent space $T_q T_pA$ splits off an $\mathbb{R}$-factor. The approach is similar to \autoref{sec: product case}, only instead of approximating $f_c$ one approximates $-\dist_{o_p}^2$.  For this define $F_q: B_{r_q}(q) \rightarrow \mathbb{R}$ by
		\begin{equation}\label{Eq: Definition of local defining function}
		F_q(x) := \vert q o_p\vert \cdot  B_q(x) + \frac{f_q^{\varepsilon}(x)}{2}+ \varepsilon  \vert xq \vert^2  -\frac{r_q^2 \cdot \varepsilon}{4}+ \frac{\vert q o_p \vert^2}{2},
		\end{equation}
		where $B_q(x)$ denotes the Busemann-function associated to the ray starting at $o_p$ and going through $q$ and $f_q^{\varepsilon}$ is the function coming from $P(n,1)$. One easily sees that $F_q$ is $(-2+\varepsilon)$-concave and admits canonical lifts. An obvious modification of the arguments in \autoref{sec: product case} implies theorem A and theorem B in the case that $A$ has no boundary.

\bibliographystyle{alphadin}
\bibliography{Bibliographie}

\end{document}